\documentclass[11pt]{amsart}

\usepackage[utf8]{inputenc}

\usepackage{amsmath}
\usepackage{amssymb}
\usepackage{amsfonts}
\usepackage{graphicx}
\usepackage{amsthm}
\usepackage{enumerate}
\usepackage{lscape}
\usepackage{dsfont}
\usepackage{color}
\usepackage{mathtools}

\usepackage{setspace}

\newtheorem{theor}{Theorem}[section]
\newcounter{tmp}

\newtheorem{dfn}{Definition}

\newtheorem{prop}[theor]{Proposition}
\newtheorem{corol}[theor]{Corollary}
\newtheorem{lem}[theor]{Lemma}
\newtheorem{guess}{Conjecture}

\theoremstyle{remark}
\newtheorem{rem}[theor]{Remark}

\newcommand{\Z}{\mathds{Z}}

\newcommand{\CP}{\mathds{C}\mathrm{P}}

\newcommand{\f}{\rightarrow}
\newcommand{\de}{\partial}
\newcommand{\dd}{\mathrm{D}}

\newcommand{\D}{\mathcal{D}}

\newcommand{\R}{\mathbb{R}}

\newcommand{\C}{\mathbb{C}}

\newcommand{\N}{\mathbb{N}}

\newcommand{\K }{K\"{a}hler\ }
\newcommand{\KE }{K\"{a}hler-Einstein\ }

\begin{document}
\title[2-dimensional KE metrics induced by complex projective spaces]{$2$--dimensional K\"{a}hler-Einstein metrics induced by finite dimensional complex projective spaces}

\author{Gianni Manno}
\address{(G. Manno) Dipartimento di Scienze Matematiche ``G. L. Lagrange'', Politecnico di Torino\\Corso Duca degli Abruzzi 24, 10129 Torino}
\email{giovanni.manno@polito.it}

\author{Filippo Salis}
\address{(F. Salis) Istituto Nazionale di Alta Matematica ``F. Severi'', Politecnico di Torino\\Corso Duca degli Abruzzi 24, 10129 Torino }
\email{filippo.salis@polito.it}

\thanks{The authors gratefully acknowledge support by the project ``Connessioni proiettive, equazioni di Monge-Amp\`ere e sistemi integrabili'' (INdAM), ``MIUR grant Dipartimenti di Eccellenza 2018-2022 (E11G18000350001)'' and PRIN project 2017 ``Real and Complex Manifolds: Topology, Geometry and holomorphic dynamics'' (code 2017JZ2SW5). Gianni Manno is a member of GNSAGA of INdAM.}

\begin{abstract}
{ We give a complete list of non-isometric} bidimensional rotation invariant \KE submanifolds of a finite dimensional complex projective space endowed with the Fubini-Study metric.  This solves  in the aforementioned case a classical and long-staying problem addressed among others in \cite{ch} and \cite{ts}.
\end{abstract}

\maketitle
\tableofcontents

\section{Introduction}

\subsection{Description of the problem and state of the art}\label{sec.description}

Holomorphic and isometric immersions (from now on \emph{\K immersions}) into  complex space forms (i.e. \K manifolds with constant holomorphic sectional curvature) are a classical topic  in complex differential geometry. Even though it has been extensively studied starting from  S. Bochner's  work \cite{bochner} and  E. Calabi's   seminal paper  \cite{Cal},
a complete classification of \K manifolds admitting such type of immersions does not exist, even for \K manifolds of great interest, such as \KE manifolds and homogeneous \K ones.

In \cite{umehara2}, M. Umehara  classified \KE manifolds that are \K immersed into a finite dimensional complex space form with \emph{non-positive} holomorphic sectional curvature: they are the totally geodesic submanifolds of either the complex Euclidean space or the complex hyperbolic one.
In the case when the space form has \emph{positive} holomorphic curvature, i.e., the complex projective space  $\CP^n$ (endowed with the Fubini--Study metric $g_{FS}$), only some partial results exist (see for instance \cite{smyth,ch,ts,hano,hulin,hulinlambda}). Motivated by this,  in the present paper we consider the problem to  list those complex manifolds  admitting a \emph{projectively induced} \KE metric.

%Conversely, we have only some partial results (see \cite{smyth}, \cite{ch}, \cite{ts}, \cite{hano}, \cite{hulin}, \cite{hulinlambda}) about the problem we  consider in the present paper, namely to  list those complex manifolds   admitting a \emph{projectively induced} \KE metric.

%Moreover, in   the last decades, the celebrated G. Tian's result \cite{tian4} (stating that every \KE metric on a Fano manifold is the $C^2$-limit of a sequence of nomalized projectively induced metrics) has contributed greatly to revive  interest  towards the case of those  metrics \K immersed into  complex space forms with \emph{positive holomorphic setional curvature}.
% in addition to . shows with explicit examples that , and

\begin{dfn}\label{def1}
We say that a \K metric on a connected complex  manifold $M$ is \emph{projectively induced}, if $M$ can be  \K immersed  into a finite dimensional\footnote{Often in the literature, the definition of projectively induced metric does not exclude that ambient complex projective space may be infinite dimensional. Our choice is dictated by purely practical reasons, indeed we are going to study a conjecture that cannot be extended to the infinite dimensional setting (see Remark \ref{infinite}).} complex projective space  $\CP^n$ endowed with the Fubini--Study metric $g_{FS}$, namely the metric associated to the \K form  given in homogeneous coordinates by $$\frac{i}{2} \de\bar\de\log\left(|Z_0|^2+\mathellipsis+ |Z_n|^2\right).$$
\end{dfn}
The most relevant  facts known so far about
%such issue
complex manifolds admitting projectively induced \KE metrics can be summarized by the following theorems:
\begingroup
\setcounter{tmp}{\value{theor}}
\setcounter{theor}{0}
\renewcommand\thetheor{\Alph{theor}}
\begin{theor}[S. S. Chern \cite{ch}, K. Tsukada \cite{ts}]\label{ct}
Let $(M,g)$ be a complete $n$-dimensional K\"ahler--Einstein manifold ($n\geq 2$). If $(M,g)$ admits a K\"ahler immersion into $(\CP^{n+2},g_{FS})$, in particular $g$ is projectively induced, then $M$ is either totally geodesic or the complex quadric  in $(\CP^{n+1},g_{FS})$.
\end{theor}
\begin{theor}[D. Hulin  \cite{hulinlambda}]
If a compact \KE manifold is projectively induced then its Einstein constant is positive.
\end{theor}
\endgroup
\setcounter{theor}{\thetmp}

%\textbf{\color{red} PERCHE' CITARE IL SEGUENTE RISULTATO QUI?}  Moreover, we have to mention also the fundamental result due to Hano about \KE submanifolds  as complete intersection in the complex projective space (see \cite{hano}).
Considering the previous results and taking also into account that all the explicit examples hitherto known are homogeneous manifolds (cfr. \cite{tak}), it has been  proposed the following conjecture (see e.g. \cite[Chap. 4]{loizedda}):

\begin{guess}\label{conj1}
If $(M,g)$ is  a \KE manifold endowed with a projectively induced metric, then it is an open subset of a complex flag manifold\footnote{A compact simply-connected \K manifold acted upon  transitivity by its holomorphic isometry group.}.
\end{guess}
\begin{rem}\label{infinite}
The conjecture cannot be extended to \KE manifolds embedded into the infinite dimensional complex projective space\footnote{The classification of \KE manifolds admitting an immersion into an infinite dimensional complex space form is an open problem in all three cases (for some partial results see e.g. \cite{symm,extr,mathz,tohoku}).}, indeed  explicit examples of such non-homogeneous \KE manifolds can be found in \cite{wallart,hao}.
\end{rem}

\subsection{Description of the main result}

% We are going to approach the problem from a different perspective compared to the past. Indeed, we do not
The present paper is  a first step toward a more ambitious research plan aimed at  approaching  the problem described in Section \ref{sec.description} (in particular, Conjecture \ref{conj1}) from a different  perspective compared to the past: we do not give any assumption  about the codimension of the studied immersions (cfr.  \cite{smyth,ch,ts,salis}). Our only assumption involves the group of symmetries of the metric.
Indeed,  our goal  will be to test the above mentioned conjecture in  the  case of {rotation invariant} \K metrics (see also \cite{extr} for a list of  projectively induced extremal metrics in the   radial\footnote{I.e. those \K metrics admitting a local potential depending only on the sum of the moduli of  certain local coordinates.} case).
\begin{dfn}
 A \K metric $g$ on a connected complex manifold $M$  is said to be \emph{rotation invariant} if  there exist a point $p\in M$, a local coordinate system $(z_1,\dots, z_n)$ centered at $p$ and a (local) K\"ahler potential $\Phi$ for $g$   such that $\Phi$  only depends on $|z_1|^2,\dots,|z_n|^2$.
 \end{dfn}

Since complex projective spaces are the only irreducible rotation invariant flag manifolds (cfr. \cite{APcoord,hamb}) and since only the integer multiples of the Fubini-Study metric  are projectively induced  (see \cite{Cal,loizedda}),  in the specific case of  rotation invariant \K metrics   Conjecture  \ref{conj1} reads as:
\begin{guess}\label{subconj}
The only projectively induced and rotation invariant \KE manifolds are open subsets of $\CP^{n_1}\times\mathellipsis\times\CP^{n_k}$ endowed with the \K metric
$$q\left(c_1 g_{FS} \oplus \mathellipsis \oplus c_k g_{FS}\right), $$
where $k$ and $q\in\Z^+$, $c_i=\frac{1}{G^{k-1}}\prod_{{ j\neq i}}(n_j+1)$ for $i=1,\mathellipsis,k$ and $G=\mathrm{gcd}(n_1+1,\mathellipsis, n_k+1)$, namely   the greatest common divisor between $n_1+1,\mathellipsis, n_k+1$.
\end{guess}
\begin{rem}\label{embedding}
%\textbf{\color{red} NON SO SE IL POSTO GIUSTO DI QUESTO REMARK E' QUESTO. VEDIAMO.}
The homogeneous spaces  $\big(\CP^{n_1}\times\dots\times\CP^{n_k},q (c_1g_{FS}\oplus\dots\oplus  c_kg_{FS})\big)$ are fully embedded into  $\CP^{{n_1+q c_1\choose q c_1}\cdots{n_k+q c_k\choose q  c_k}-1}$. A \K embedding can be explicitly described through a composition of suitable normalizations of the Veronese embeddings:
\begin{align*}
(\CP^{n},cg_{FS})& \to (\CP^{\binom{n+c}{c}-1},g_{FS}) \\
[Z_i]_{0\leq i\leq n}& \mapsto\sqrt\frac{(c-1)!}{c^{c-2}} \left[\frac{Z_0^{c_0}\mathellipsis Z_n^{c_n}}{\sqrt{c_0!\mathellipsis c_n!}} \right]_{c_0+\mathellipsis +c_n=c},
\end{align*}
together with a {Segre embedding}  (cfr. \cite{Cal,loizedda}).

%\begin{align*}
%(\CP^{n_1}\times\CP^{n_2},g_{FS}\oplus g_{FS})& \to (\CP^{(n_1+1)(n_2+1)-1},g_{FS}) \\
%([z_i]_{0\leq i\leq n_1},[w_j]_{0\leq j\leq n_2})& \mapsto [z_iw_j]_{ (i,j)\in\{0,\mathellipsis, n_1\}\times\{0,\mathellipsis, n_2\}}.
%\end{align*}
 \end{rem}

Our main result  is  { contained in the following theorem, that solves Conjecture \ref{subconj} in the $2$-dimensional case.}

\begin{theor}\label{mainth}
If $(M,g)$ is a $2$-dimensional \KE manifold whose metric is rotation invariant and projectively induced, then $(M,g)$  is an open subset of either $(\CP^2,q\, g_{FS})$ or \mbox{$\left(\CP^1\times\CP^1,q(g_{FS}\oplus g_{FS})\right)$}, where $q\in\Z^+$.
\end{theor}

\section{Proof of Theorem \ref{mainth}}

%\subsection{\textbf{\color{red}Idea of the proof and a new approach to the problem?}}
The proof of Theorem \ref{mainth} is organized in three subsections, described below.

In Section \ref{diastasis}, we recall the definition of Calabi's diastasis function and Bochner's coordinates.

% we find a one-to-one correspondence between  Calabi's diastasis functions\footnote{A particular \K potential uniquely determined by the correspondent real-analytic \K metric. Its definition and some properties will be recalled at the begin of the next section.} related to  rotation invariant  projectively induced metrics and real polynomials with fixed constant and linear terms. \textbf{NON MI SEMBRA}

In Section \ref{realMA}, on account of the results recalled in Section \ref{diastasis}, by proving several auxiliary lemmas, we  rephrase in Proposition \ref{newmain} the statement of Theorem \ref{mainth}  in terms of existence and uniqueness of polynomial solutions of a particular family of real Monge-Ampère equations, where the unknown function is the Calabi's diastasis function and the independent variables are the moduli of the Bochner's coordinates. The existence of polynomial solutions is a part of Proposition \ref{newmain}, whereas the proof of the uniqueness of such solutions is the core of Section \ref{sec.proof.prop}.

In fact, in Section \ref{sec.proof.prop}, we find a set of suitable initial conditions for the aforementioned family of Monge-Ampère equations: an arbitrary polynomial solution to a Monge-Ampère equation of this family needs to satisfy one and only one initial condition of such set. Taking this into account, in the end of the section, we prove that the solutions we listed in Proposition \ref{newmain} are actually unique, thus getting the statement of Theorem \ref{mainth}.

%By taking into account the algebraic nature of  such kind of solutions, we are going to identify the Cauchy data that they need to satisfy on the coordinate axes {\textbf{\color{red}QUESTA FRASE E' UN PO' AMBIGUA. FORSE IDENTIFY--CHARACTERIZE? NON SO}. Once proven that  these conditions  are not characteristic,  it follows from the Cauchy-Kowalevsky Theorem that an analytic solution satisfying them is formally unique. It will be enough to prove that the polynomials related to  $(\CP^2,q g_{FS})$ and  $\left(\CP^1\times\CP^1,q(g_{FS}\oplus g_{FS})\right)$ are the only solutions in the  class we are studying. {\textbf{\color{red}ANCHE QUI E' UN PO' AMBIGUO. NE POSSIAMO DISCUTERE INSIEME} Hence we get our statement. {\textbf{\color{red}, thus getting the statement of Theorem \ref{mainth}}}

\subsection{Calabi's diastasis function}\label{diastasis}
In order to prove Theorem \ref{mainth}, we need to recall the definition of Calabi's diastasis function and some of its properties.

Let $(M,g)$  be a \K manifold with a local \K potential $\Phi$, namely $\omega=\frac{i}{2}\de\bar\de\Phi$,
 where $\omega$ is the \K form associated to $g$. If $g$ (and hence $\Phi$) is assumed to be real analytic, by duplicating the variables $z$ and $\bar z$, $\Phi$ can be complex analytically extended to a function $\tilde \Phi$ defined in a neighbourhood $U$ of the diagonal containing $(p,\bar p)\in M\times \bar M$ (here $\bar M$ denotes the manifold conjugated to $M$). Thus one can consider the power expansion of $\Phi$ around the origin  with respect to $z$ and $\bar z$ and write it as
\begin{equation}\label{powexdiastc}
\Phi(z,\bar z)=\sum_{j,l=0}^{\infty}a_{jl}z^{m_j}\bar z^{m_l},
\end{equation}
where we arrange every $n$-tuple of nonnegative integers as a sequence \linebreak$m_j=(m_{j,1},\dots,m_{j,n})$ and order them as follows: $m_0=(0,\dots,0)$ and if $|m_j|=\sum_{\alpha=1}^n m_{j,\alpha}$, $|m_j|\leq |m_{j+1}|$ for all positive integer $j$. Moreover, $z^{m_j}$ denotes  the monomial in $n$ variables $\prod_{\alpha=1}^n z_\alpha^{m_{j,\alpha}}$.

 A \K potential is not unique, but it is defined up to an addition of the real part of a holomorphic function. The \emph{diastasis function} $\dd_0$ for $g$ is nothing but the \K potential around $p$ such that each matrix $(a_{jk})$ defined according to equation (\ref{powexdiastc}) with respect to  a coordinate system $z=(z_1,\dots,z_n)$ centered in $p$, satisfies $a_{j0}=a_{0j}=0$ for every nonnegative integer $j$.

Moreover, for any real analytic K\"ahler manifold there exists a coordinates system, in a neighbourhood of each point,  such that
\begin{equation}\label{bochnercoordinates}
\dd_0(z)=\sum_{\alpha=1}^n|z_\alpha|^2+\psi_{2,2},
\end{equation}
where $\psi_{2,2}$ is a power series with degree $\geq 2$ in both $z$ and $\bar z$. These coordinates, uniquely determined up to unitary transformation (cfr. \cite{bochner,Cal}), are called \emph{Bochner's coordinates} (cfr. \cite{bochner,Cal,hulin,hulinlambda,ruan,tian4}).

Notice that throughout this paper we will consider either projectively induced metrics or \KE metrics. In both cases these metrics are real analytic and hence diastasis functions and Bochner's coordinates are defined. Moreover, in the particular case of rotation invariant metrics, the diastasis function around the origin of the Bochner's coordinates system is a rotation invariant \K potential.\\

\subsection{Real Monge-Ampère equations}\label{realMA}
The  lemmas contained in this section hold for manifolds of arbitrary dimension. By applying them to the bidimensional case, we  show how the property of the projectively induced metrics to be rotation invariant,  allows us to address Conjecture \ref{subconj} through real analysis' techniques. Indeed, we prove the equivalence of the statement of Theorem \ref{mainth} to a uniqueness problem in a class of  solutions of a family of real Monge-Ampère equations (Proposition \ref{newmain}).

\begin{lem}\label{diastpolin}
Let $V$ be an open subset of $\C^n$ where it is defined   a  rotation invariant potential for a \K metric $g$. Let  $f: (V,g)\f (\CP^N,g_{FS})$ be a full\footnote{A holomorphic  immersion $f\!:U\f\CP^n$ is said to be \emph{full} provided $f(U)$ is not contained in any $\CP^h$ for $h< n$.} \K immersion.
Then $\dd_0(z)$  can be written  as
\begin{equation}\label{diastbochpi}
\dd_0(z)=\log\left(P(z)\right),
\end{equation}
where
\begin{equation}\label{eq.Pz}
P(z)=1+\sum_{j=1}^n|z_j|^2+\sum_{j=n+1}^N a_j|z^{m_{h_j}}|^2
\end{equation}
with $a_j>0$ and $h_j\neq h_l$ for $j\neq l$.
\end{lem}
\begin{proof}
%Let $V$ be an open subset of $M$ where is defined a local coordinate system $z=(z_1,\mathellipsis,z_n)$ centered at $p$, such that there exists a local \K potential for $g$, rotation invariant with respect to $z$.
% Indeed, the existence of an open subset of a connected manifold $M$ embedded into $\CP^N$ implies the existence for every point  $p\in M$ of a neighborhood of $p$ embedded into the same complex projective space  \cite[Theorem 10]{Cal}.
Recall that $Z_0,\dots, Z_N$ are the homogeneous coordinates on $\CP^N$ (see Definition \ref{def1}).
Up to a unitary transformation of $\CP^N$ and by shrinking $V$ if necessary we can assume $f(p)=[1, 0\dots, 0]$ and $f(V)\subset U_0=\{Z_0\neq 0\}$. Since the affine coordinates on $U_0$ are Bochner's coordinates for the Fubini--Study metric $g_{FS}$,   by \cite[Theorem 7]{Cal}, $f$
can be written as:
$$f:V\f \C^N\,,\quad z=(z_1, \dots , z_n)\mapsto (z_1, \dots z_n, f_{n+1}(z), \dots , f_N(z)),$$
where
$$f_j(z)=\sum_{l=n+1}^{\infty}\alpha_{jl}z^{m_l}, \ j=n+1, \dots, N.$$
Since the diastasis function is hereditary (see \cite[Prop. 6]{Cal} ) and that of  $\CP^n$
around the point $[1, 0\dots, 0]$ is given on $U_0$ by $\Phi(z)=\log(1+\sum_{j=1}^N|z_j|^2)$, where $z_j=\frac{Z_j}{Z_0}$,  one gets
$$\dd_0(z)=\log\left(1+\sum_{j=1}^n|z_j|^2+\sum_{j=n+1}^N|f_j(z)|^2\right).$$
The rotation invariance of $\dd_0(z)$ and the fact that $f$ is full  imply that the $f_j$'s are monomials
of $z$ of different degree and formula \eqref{diastbochpi} follows.
\end{proof}
By setting
\begin{equation}\label{eq.xz}
x=(x_1,\mathellipsis,x_n)=(|z_1|^2,\mathellipsis,|z_n|^2)\,,
\end{equation}
the diastasis function $\dd_0$ of a rotation invariant \K metric $g$ can be viewed as a function of the real variables $x_i$.

\smallskip\noindent
From now on we set, with a little abuse of notation,
\begin{equation}\label{eq.Px}
P(x)=P(z(x))\,,
\end{equation}
where $P(z)$ is given by \eqref{eq.Pz} and $x$ by $\eqref{eq.xz}$.

\smallskip\noindent
A diastasis function of a rotation invariant \KE metric satisfies the following lemma.
\begin{lem}\label{realma}
If $g$ is a rotation invariant \KE metric, its diastasis $\dd_0(x)$, where $x$ is given by \eqref{eq.xz},
is a solution of the real Monge-Ampère equation
\begin{equation}\label{eq.real.MAE}
\det\left( \frac{\partial^2 \dd_0}{\partial{x_\alpha}\partial{x_\beta}}x_\alpha+\frac{\partial \dd_0}{\partial{x_\alpha}}\delta_{\alpha\beta}\right)=e^{-\frac{\lambda}{2}\dd_0}
\end{equation}
where $\delta_{\alpha\beta}$ is the Kronecker delta and $\lambda$ is the Einstein constant.
\end{lem}
\begin{proof}
A K\"ahler metric $g$ with diastasis function $\dd_0(z)$  is  Einstein  (see e.g. \cite{Mor}) if and only if  there exists $\lambda\in\R$ such that
$$\lambda\frac{i}{2} \de\bar\de \dd_0=-i\de\bar\de\log\det(g_{\alpha\bar\beta}).$$
Hence, by the $\partial\bar\partial$-lemma,  there exists a holomorphic function $\varphi$ such that
\begin{equation}\label{eq.complex.MAE}
\det(g_{\alpha\bar\beta})=e^{-\frac{\lambda}{2}(\dd_0+\varphi+\bar \varphi)}.
\end{equation}
Once Bochner's coordinates are set, by comparing the series expansions of both sides of the previous equation, we get that $\varphi+\bar \varphi$ is forced to be zero (cfr. \cite{note,hulinlambda,salis}). The PDE \eqref{eq.complex.MAE}, in coordinates \eqref{eq.xz}, coincides with \eqref{eq.real.MAE}.
\end{proof}

\begin{lem}\label{bound}
The Einstein constant $\lambda$ of a projectively induced and rotation invariant \KE manifold of dimension $n$ is a positive rational number less than or equal to $2(n+1)$.
\end{lem}
\begin{proof}
By Lemma \ref{diastpolin}, the diastasis of a rotation invariant and projectively induced \K metric can be written as $\dd_0(x)=\log(P(x))$, where $P$ is a polynomial of type \eqref{eq.Px}.
By Lemma \ref{realma}, we have
\begin{equation}\label{MA*}
\D_n(P)=P^{-\frac{\lambda}{2}+n+1},
\end{equation}
where
we denote by $\D_n$ the following differential operator
\begin{equation*}\label{operator}
\D_n(P)=\frac{\det\left[\left(P \frac{\partial^2 P}{\partial{x_\alpha}\partial{x_\beta}}-\frac{\partial P}{\partial{x_\alpha}}\frac{\partial P}{\partial{x_\beta}}\right)x_\alpha+P\frac{\partial P}{\partial{x_\alpha}} \delta_{\alpha\beta}\right]_{1\leq\alpha,\beta\leq n}}{P^{n-1}}.
\end{equation*}
By multilinearity of determinants and by considering that
%$\left(P_{\alpha}P_{\beta}x_\alpha\right)_{1\leq \alpha,\beta\leq n}$
$\left(\frac{\partial P}{\partial{x_\alpha}}\frac{\partial P}{\partial{x_\beta}}x_\alpha\right)_{1\leq \alpha,\beta\leq n}$
is a rank-1 matrix, we get that left side of \eqref{MA*}  is a polynomial. Therefore $\lambda$ needs to be a rational number satisfying the inequality  $-\frac{\lambda}{2}+n+1\geq 0$. Then we obtain the upper bound for the Einstein constant $\lambda$.
Furthermore, by comparing the degrees of both sides of \eqref{MA*}, we get  $\lambda\geq 2\frac{n}{\deg P }>0$.
\end{proof}
\begin{rem}
It is worth pointing out that Conjectures \ref{conj1} and \ref{subconj} are of local nature, i.e. there are no the topological assumptions on projectively induced manifolds and the  immersions are not required to be injective. Moreover, we notice that  Lemma \ref{bound} has  important topological  consequences supporting Conjecture \ref{subconj}, namely \emph{a rotation invariant and projectively induced \KE manifold is an open subset of a complete, compact and simply-connected \KE manifold globally embedded into a finite dimensional complex projective space}. Indeed, every \KE manifold embedded into a (possibly infinite dimensional) complex projective space can be extended to a \emph{complete} \KE manifold  $M$ (see \cite{hulin}). Since the Einstein constant of $M$ is positive  by Lemma \ref{bound}, then $M$ is \emph{compact}  by Myers' theorem. Moreover, $M$ is a \emph{simply connected} by a well-known theorem of Kobayashi \cite{koricci} and every local immersion of a simply-connected manifold into a complex space form can be extended to a global one (see \cite{Cal}).
\end{rem}

{Now, let $\lambda$ be the Einstein constant of a projectively induced and rotation invariant \KE manifold of dimension $n$. In view of Lemma \ref{bound},}
$\lambda=2 \frac{s}{q}$, where $\gcd(s,q)=1$. Since $\gcd(2nq,s)=1$, a polynomial solution of type \eqref{eq.Px} to \eqref{MA*}, is forced to be the $q$-th power of a polynomial $R(x)$. After the change of variables $ x=\frac{\tilde x}{q}$, we easily check that $R(\tilde x)$ is a solution for \eqref{MA*}  with $q=1$. Vice versa, every solution $R(\tilde x)$  of \eqref{MA*} for $q=1$ gives rise to a solution of \eqref{MA*} for $q\neq1$ by taking the  $q$-th power of $R(\tilde x)$ and by considering the same changing of variables  $\tilde x=q x$. Hence,  we are going  to study from now on the real Monge-Ampère equations \eqref{MA*} just when $q=1$.

By restricting  \eqref{MA*} to the case $n=2$, by recalling that, for our purposes, we consider only solutions %we are interested only in  solutions
belonging to the polynomial class \eqref{eq.Px} and that the upper bound for the above parameter $s$ can be obtained by Lemma \ref{bound},
 we have that the statement of Theorem \ref{mainth} can be get  by proving the following proposition.
 \begin{prop} \label{newmain}
  The only solutions of  type
\begin{equation}\label{soltype}
P(x)=P(x_1,x_2)=1+x_1+x_2+\xi(x_1,x_2),
\end{equation}
where $\xi$ is a polynomial with positive coefficients and no terms of degree less than 2,  to the real Monge-Ampère equation
\begin{equation}\label{MAn2}
\D_2(P)=P^{3-s}
\end{equation}
for some integer $s\in\{1,2,3\}$, are
\begin{equation}\label{solution}\begin{cases}
1+x_1+x_2,& \text{when }s=3;\\
(1+x_1)(1+x_2)& \text{when }s=2;\\
(1+\frac{x_1+x_2}{3})^3 \text{ and } (1+\frac{x_1}{2})^2(1+\frac{x_2}{2})^2& \text{when } s=1.\\
\end{cases}\end{equation}
\end{prop}

\subsection{Proof of Proposition \ref{newmain}}\label{sec.proof.prop}
As a first step towards the proof of Proposition \ref{newmain}, we characterize  the initial conditions that an arbitrary polynomial solution  of type \eqref{soltype}  to the Monge-Ampère equation \eqref{MAn2} needs to satisfy on the coordinate axes. These conditions will be given by the Corollary \ref{Cauchydata} of the following lemma, that holds true for  any dimension.
\begin{lem}\label{pt}
The restriction $p$ on a coordinate axis  of a polynomial solution of type \eqref{eq.Px} to the Monge-Ampère equation \eqref{MA*} reads as:
\begin{equation}\label{cauchycond}
\begin{cases}
 p(t)=1+t,& \text{when }s=n+1;\\ %;\ q(x)=1
 p(t)=\left(1+\frac{t}{k}\right)^k,\text{ with } k\in\{1, 2\}& \text{when }s=n;\\%;\ q(x)=\left(1+\frac{x}{k}\right)^{2-k}
  p(t)=\left(1+\frac{t}{k}\right)^k,\text{ with } k\in\Z^{+} & \text{when }1\leq s \leq n-1.\\%;\ q(x)=\left(1+\frac{x}{k}\right)^{k(n-s-1)+2}
\end{cases}
\end{equation}
\end{lem}
\begin{proof}
Let $p$ be the restriction on the $i$-th coordinate axis (i.e. the line $x_j=0,\text{ for } j\neq i$) of  a polynomial solution $P$ of type \eqref{eq.Px} to the Monge-Ampère equation \eqref{MA*}. Hence, we have that
\begin{equation}\label{MArestr}
\D_1\left(p(t)\right) q(t) =p(t)^{n-s+1},
\end{equation}
where the polynomial $q(t)$ is the restriction on the $i$-th coordinate axis of  $\prod_{j\neq i}\frac{\de P}{\de x_j}$.
Let $\{-r_1,\mathellipsis,-r_R\}$ be the (possibly complex) distinct roots  of  $p$, namely\footnote{Notice that the constant term of $p(x)$ and $q(x)$ are fixed to be equal to 1 by the definition of \eqref{eq.Px}.}
$$p(t)=\frac{1}{\prod_{i=1}^R r_i^{k_i}}\prod_{i=1}^R (t+ r_i)^{k_i}.$$
Considering that
\begin{equation*}
 \D_1\left( \prod_{i=1}^R (t+ r_i)^{k_i}\right)= \prod_{i=1}^R (t+ r_i)^{2k_i-2}\sum_{i=1}^R k_i r_i  \prod_{\substack{j=1\\ j\neq i}}^R (t+ r_j)^{2},
\end{equation*}
the equation \eqref{MArestr} can be written as
\begin{equation*}
\left( \sum_{i=1}^R k_i r_i  \prod_{\substack{j=1\\ j\neq i}}^R (t+ r_j)^{2}\right) q(x) =\frac{1}{\prod_{i=1}^R r_i^{k_i(n-s-1)}}\prod_{i=1}^R (t+ r_i)^{k_i(n-s-1)+2}.
\end{equation*}
Therefore we get
 \begin{equation}\label{cauchycond2}
 q(t) =\frac{1}{\prod_{i=1}^R r_i^{k_i(n-s-1)+2}}\prod_{i=1}^R (t+ r_i)^{k_i(n-s-1)+2}
 \end{equation}
 and
\begin{equation}\label{eq:system.Filippo}
 \sum_{i=1}^R k_i r_i  \prod_{\substack{j=1\\ j\neq i}}^R (t+ r_j)^{2}-\prod_{i=1}^R  r_i^{2}=0.
\end{equation}
Let now consider \eqref{eq:system.Filippo} as a linear system in the variables $k_1,\mathellipsis,k_R$. If $R=1$, such system consists of just one equation, which  has a unique solution: $$k_1=r_1.$$
If $R\geq 2$ it cannot be compatible for any $t$. Indeed, being the left hand side of \eqref{eq:system.Filippo} a polynomial in $t$ of degree $2R-2$,  in particular its first $R$ higher order coefficients have to vanish. Therefore, $k_1,\mathellipsis,k_R$ need to satisfy a homogeneous system, whose  determinant of the coefficients matrix can be easily computed:
$$
R!\prod_{i=1}^R r_i\prod_{1\leq i<j\leq R}(r_i-r_j).
$$
In view of our hypotheses, such determinant is always different from zero. Therefore  our system  admits only the trivial solution, leading to a contradiction, since $k_i$  represent the multiplicity of a root of a polynomial, so  they should be positive.
\end{proof}
\begin{corol}\label{Cauchydata}
Any arbitrary polynomial solution of type \eqref{soltype} to the Monge-Ampère equation  \eqref{MAn2} satisfies  one and only one of the following initial conditions on the coordinate axis $x_2=0$:
\begin{equation}\label{Cauchy}
\begin{cases}
P(x_1,0)=1+x_1,\ \frac{\de P}{\de x_2}(x_1,0)=1 & \text{when }s=3;\\
P(x_1,0)=1+x_1,\ \frac{\de P}{\de x_2}(x_1,0)=1+x_1& \text{when }s=2;\\
P(x_1,0)=\left(1+\frac{x_1}{2}\right)^2,\ \frac{\de P}{\de x_2}(x_1,0)=\left(1+\frac{x_1}{2}\right)^2 &\qquad \text{ or } \\
P(x_1,0)=\left(1+\frac{x_1}{3}\right)^3,\ \frac{\de P}{\de x_2}(x_1,0)=\left(1+\frac{x_1}{3}\right)^2& \text{when } s=1.\\
\end{cases}\end{equation}
\end{corol}
\begin{proof}
Let $P$ be  a solution  of type \eqref{soltype} to \eqref{MAn2}. By Lemma \ref{pt}, $P(x_1,0)=\left(1+\frac{x_1}{k} \right)^k$ and $P(0, x_2)=\left(1+\frac{x_2}{h} \right)^h$ for suitable $k$, $h\in\Z^+$. Moreover, by  \eqref{cauchycond2}, $\frac{\partial P}{\partial x_2}(x_1,0)=\left(1+\frac{x_1}{k} \right)^{k(1-s)+2}$  and $\frac{\partial P}{\partial x_1}(0,x_2)=\left(1+\frac{x_2}{h} \right)^{h(1-s)+2}$. By computing $\frac{\partial^2 P}{\partial x_1\partial x_2}(0,0)$, we get $k=h$. Therefore, $P$ reads as:
\begin{multline}\label{soln2}
\left(1+\frac{x_1}{k}\right)^k+\left(1+\frac{x_2}{k}\right)^k-1+x_1\left(1+\frac{x_2}{k}\right)^{k(1-s)+2}+x_2\left(1+\frac{x_1}{k}\right)^{k(1-s)+2}\\
-x_1-x_2-\left(1-s+\frac{2}{k} \right)x_1x_2+x_1^2x_2^2\ \eta(x_1,x_2),
\end{multline}
where $\eta$ is a polynomial.
By putting \eqref{soln2} in  \eqref{MAn2}, by differentiating both sides of the equation by $\frac{\partial^2}{\partial x_1\partial x_2}$ and by evaluating at $(0,0)$,
%the second order mixed derivative and evaluating at the origin,
we straightforwardly get the following Diophantine equation
$$s^2k^2-5sk+6=0.$$
Therefore,  by solving the previous equation,  we easily  get our statement.\end{proof}

Since each solution \eqref{solution} satisfies the correspondent initial condition  \eqref{Cauchy}, we   conclude the proof of Proposition \ref{newmain} by showing that
%\footnote{Notice that we cannot apply the Cauchy-Kowalevsky Theorem because the Cauchy data \eqref{Cauchy} are characteristic.}:
\begin{lem}
If there exists a  polynomial solution to \eqref{MAn2}  satisfying an initial condition of type \eqref{Cauchy}, then it is unique.
\end{lem}
\begin{proof}
Let $F_s$ be a function whose zero defines the PDE \eqref{MAn2}, i.e.,  $F_s:=\D_2(P)-P^{3-s}$. Then,
from a straightforward computation, we  get the following formula
\begin{multline}\label{unique}
\frac{\partial^h F_s}{\partial x_2^h} (x_1,0)=\\ \left( h\left( P\frac{\partial^2 P}{\partial x_1^2}x_1-\left(\frac{\partial P}{\partial x_1}\right)^2x_1+P\frac{\partial P}{\partial x_1}\right)\frac{\partial^{h+1} P}{\partial x_2^{h+1}} + T^h\right)(x_1,0)\,,
\end{multline}
where $T^h(x_1,0)$  is a polynomial expression in $x_1$, $P(x_1,0)$ and derivatives of $P$ up to order $h+1$ (computed in $(x_1,0)$), that does not contain $\frac{\partial^h P}{\partial x_2^h}(x_1,0)$ and  $\frac{\partial^{h+1} P}{\partial x_2^{h+1}}(x_1,0)$.
If $P$ is a  polynomial solution to \eqref{MAn2} satisfying an initial condition of type \eqref{Cauchy},   $P(x_1,0)=\left(1+\frac{x_1}{k}\right)^k$ for a suitable integer $k$, hence we have
$$\left( P\frac{\partial^2 P}{\partial x_1^2}x_1-\left(\frac{\partial P}{\partial x_1}\right)^2x_1+P\frac{\partial P}{\partial x_1}\right)(x_1,0)=\left( \frac{x_1}{k}+1\right)^{2k-2}\not\equiv0.$$
By considering formula \eqref{unique} when $h=1$, we realize that  initial conditions \eqref{Cauchy} uniquely determine $\frac{\partial^{2} P}{\partial x_2^{2}} (x_1,0)$, from which one obtains $\frac{\partial^{2+h} P}{\partial x_1^h\partial x_2^{2}} (x_1,0)$ for every $h\in \N$. By iteration, we get  the whole Taylor expansion of $P$ on the line $x_2=0$. Therefore, we get the statement of the lemma.
\end{proof}

\small{}

\end{document}